\documentclass{article}
\usepackage{amsfonts}
\usepackage{bm}
\usepackage{amsmath}
\usepackage{amsthm}

\usepackage[numbers,sort&compress]{natbib}
\bibpunct[, ]{[}{]}{,}{n}{,}{,}
\makeatletter
\def\NAT@def@citea{\def\@citea{\NAT@separator}}
\makeatother

\theoremstyle{plain}
\newtheorem{theorem}{Theorem}[section]
\newtheorem{lemma}[theorem]{Lemma}
\newtheorem{corollary}[theorem]{Corollary}

\theoremstyle{definition}

\theoremstyle{remark}
\newtheorem{remark}{Remark}

\begin{document}

\begin{center}
{\bf \LARGE Comparison of Nonlocal Nonlinear Wave Equations in the Long-Wave Limit}
\\ ~ \\ \vspace*{20pt}
H. A. Erbay$^{1}$, S. Erbay$^{1}$,  A. Erkip$^{2}$
\vspace*{10pt}

$^{1}$Department of Natural and Mathematical Sciences, Faculty of Engineering, Ozyegin University,  Cekmekoy 34794, Istanbul, Turkey
\vspace*{10pt}

$^{2}$Faculty of Engineering and Natural Sciences, Sabanci University,  Tuzla 34956,  Istanbul,    Turkey

\end{center}

\let\thefootnote\relax\footnote{E-mail:   husnuata.erbay@ozyegin.edu.tr, saadet.erbay@ozyegin.edu.tr, \\
albert@sabanciuniv.edu}

\begin{abstract}
    We consider  a general class of  convolution-type  nonlocal wave equations  modeling bidirectional nonlinear wave propagation.  The model involves two  small positive parameters  measuring the relative strengths of the nonlinear and dispersive effects.  We take two different kernel functions that have  similar dispersive   characteristics  in the long-wave limit and compare the corresponding solutions of the Cauchy problems  with the same initial data.   We  prove rigorously that  the difference between the two solutions remains small over a long time interval in a suitable Sobolev norm. In particular our results show that, in the long-wave limit, solutions of such nonlocal equations can be well approximated by those of  improved Boussinesq-type equations.
\end{abstract}
\vspace*{10pt}

   \vspace*{10pt}
   
    \vspace*{10pt}

\setcounter{equation}{0}
\section{Introduction}\label{sec1}

In the present work we start with the  nonlocal wave equation:
\begin{eqnarray}
  && u_{tt}=\beta_{\delta }\ast ( u+\epsilon^{p} u^{p+1})_{xx},  ~~~~ x\in \mathbb{R}, ~~~t>0, \label{nw} \\
  && u(x,0)=u_{0}(x), ~~~~ u_{t}(x,0)=u_{1}(x),  ~~~~ x\in \mathbb{R}, \label{inidata}
\end{eqnarray}%
where $u=u(x,t)$ is a real-valued function, $\epsilon $ and $\delta $ are two small positive parameters measuring the effect of nonlinearity and the effect of dispersion, respectively, $p$ is a positive integer, the symbol $\ast$ denotes convolution in the $x$-variable and $\beta_{\delta }(x) =\frac{1}{\delta }\beta (\frac{x}{\delta })$ where $\beta$ is a kernel function.  The two parameter family (\ref{nw}) is obtained from a single equation  given with the kernel  $\beta(x)$ through a suitable scaling.  Our main purpose  is to investigate the dependence of the solutions of (\ref{nw})-(\ref{inidata}) on the kernel $\beta$ in the long-wave limit.  To that end, we take two different kernel functions for which the corresponding convolution operators are elliptic (and bounded) and compare the corresponding solutions in the case of the same initial data.  We  prove that,  over a long time interval of length of order $1/\epsilon^{p}$,  the difference between the two solutions  is of order $\delta^{2k}$ (uniformly in $\epsilon$) where the integer $k$ is determined by the dispersive behaviors of the two kernels.

The nonlocal equation (\ref{nw}) describes the one-dimensional motion of a nonlocally and nonlinearly elastic medium and $u$ represents the strain (we refer the reader to \cite{Duruk2010} for a detailed description of the nonlocally and nonlinearly  elastic medium).  The dispersive nature of the elastic medium  described by (\ref{nw}) is revealed by the kernel $\beta_{\delta}$ and the parameter $\delta$ is a measure of wavenumber so that smaller $\delta$ implies longer wavelength. The parameter $\epsilon$ assures that the solutions exist over a long time interval of length of order $1/\epsilon^{p}$ (see \cite{Erbay2018}). For particular choices of the kernel function, (\ref{nw}) involves many well-known nonlinear wave equations. Two well-known examples that appear as  model equations for various physical problems are   the improved Boussinesq  equation
\begin{equation}
    u_{tt}-u_{xx}-\delta^{2}u_{xxtt}=\epsilon^{p} (u^{p+1})_{xx}  \label{ib}
\end{equation}%
corresponding to the exponential kernel $\beta_{\delta}(x)= \frac{1}{2\delta}e^{- \left\vert x\right\vert/\delta}$  and the classical elasticity equation
\begin{equation}
   u_{tt}-u_{xx}= \epsilon^{p} (u^{p+1})_{xx} \label{clas}
\end{equation}%
corresponding to the Dirac measure. We refer the reader to \cite{Duruk2010} for other examples of the kernels widely used in continuum mechanics.

A relevant question  is how the choice of the kernel function of (\ref{nw}) affects solutions. In that respect one needs a measure for the closeness of two kernels. In this work we show that the dispersive characteristics of  the kernel in the long-wave limit provides a suitable measure. We note that in the long-wave limit the Fourier modes are concentrated about the small wave numbers and  therefore  the dispersive nature of (\ref{nw}) in the small wavenumber regime is related to the Taylor expansion of the Fourier transform of the kernel around zero. As the Taylor coefficients are determined by the moments of the kernel, we can explicitly classify "close" kernels by comparing their moments and  investigate how the family of solutions of (\ref{nw})  depend on the kernel. To be more precise we take two kernels $\beta^{(1)}$ and $\beta^{(2)}$ with the same  moments up to order $2k$  and consider  the corresponding solutions $u_{1}^{\epsilon,\delta}$ and $u_{2}^{\epsilon,\delta}$  of (\ref{nw})-(\ref{inidata}) with the same initial values. We then prove that,  for a suitable norm, $\Vert u_{1}^{\epsilon,\delta}(t)-u_{2}^{\epsilon,\delta}(t)\Vert$  is of order $\delta^{2k}$ over a long time interval of length of order $1/\epsilon^{p}$.  This shows that for varying kernels with the same dispersive nature, solutions of (\ref{nw}) approximate each other  and the approximation errors originate from the dispersive nature of the kernels rather than from their shapes.   We note that, in the terminology of some authors,  our long-time existence results \cite{Erbay2018} together with the work presented here  are in fact consistency, existence, convergence results for the nonlocal bidirectional approximations of the nonlocal equation (\ref{nw}). We refer to \cite{Bona2005,Constantin2009,Lannes2012,Lannes2013,Duchene2015} and the references therein for a detailed discussion of these concepts.

This work differs from most previous studies available in the literature for several reasons; it does not focus on  unidirectional solutions and it compares solutions of two families of equations of the same type. In the literature, there have been a number of  works comparing solutions of  a parent equation  with those of a model equation describing the unidirectional propagation of long waves. The most well-known examples of one-dimensional model equations are the Korteweg-de Vries (KdV) equation \cite{Korteweg1895}, the Benjamin-Bona-Mahony (BBM) equation \cite{Benjamin1972}, the Camassa-Holm (CH) equation  \cite{Camassa1993} and the Fornberg-Whitham (FW) equation \cite{Fornberg1978}.   For a comprehensive treatment of the methods on comparisons of solutions between those unidirectional equations and the parent equations (for instance, in the context of a shallow water approximation) we refer to \cite{Craig1985,Schneider2000,Alazman2006} for the case of the KDV and the BBM equations  and to \cite{Constantin2009,Duchene2015} for the case of the CH equation (see, for instance, \cite{Gallay2001,Lannes2012} for comparisons of two-dimensional model equations). In a recent study \cite{Erbay2016}, the present authors have made  similar comparisons  between (\ref{nw}) in the context of nonlocal elasticity and the CH equation.  For emphasis, we remind the reader that all these studies consider unidirectional approximations of nonlinear dispersive  equations  whereas the present work is about bidirectional approximations. Another strength of this work lies in its level of generality because it poses minimal restrictions on the kernel.

The structure of the paper is as follows. Section \ref{sec2} is devoted to preliminaries. In Section \ref{sec3} we state the moment conditions to be satisfied by the kernels, prove the main result that establishes the estimate on the difference between two families of solutions, and  discuss how our general result can be applied to certain cases.

Throughout this paper, we use the standard notation  for  function spaces. The Fourier transform $\widehat u$ of $u$ is defined by $\widehat u(\xi)=\int_\mathbb{R} u(x) e^{-i\xi x}dx$. The $L^p$ ($1\leq p <\infty$)  norm of $u$ on $\mathbb{R}$ is represented by  $\Vert u\Vert_{L^p}$. To denote the inner product of $u$ and $v$ in $L^2$, the symbol $\langle u, v\rangle$ is used.   The notation  $H^{s}=H^s(\mathbb{R})$ is used to denote the $L^{2}$-based Sobolev space of order $s$ on $\mathbb{R}$,  with the norm $\Vert u\Vert_{H^{s}}=\big(\int_\mathbb{R} (1+\xi^2)^s \vert \widehat u(\xi)\vert^2 d\xi \big)^{1/2}$.  All integrals in this paper extend over the whole real line and the limits of integration will not be explicitly written. $C$ is a generic positive constant.  Partial differentiations are denoted by  $D_{x}$ etc.

\section{Preliminaries}\label{sec2}

In this section we give a brief discussion of the nonlocal equation and provide some preliminaries. We first note that the family of equations (\ref{nw}) can be obtained from the single   equation
\begin{equation}
    u_{tt}=\beta \ast (u+\epsilon^{p}u^{p+1})_{xx},  \label{origin}
\end{equation}%
under the change of variables $(x,t)\to (x/\delta, t/\delta)$. The nonlocal equation (\ref{origin}) with a general even kernel $\beta$ was proposed in \cite{Duruk2010} to model longitudinal motions in nonlinear nonlocal elasticity, in terms of non-dimensional quantities.

In \cite{Duruk2010}, local well-posedness of the Cauchy problem for (\ref{origin}) (and hence for (\ref{nw})) was  proved  under the regularization assumption
\begin{equation}
    0\leq \widehat{\beta }(\xi )\leq C \left( 1+\xi ^{2}\right) ^{-r/2} \label{order}
\end{equation}%
for some $r\geq 2$. In this case (\ref{origin}) becomes an $H^{s}$-valued ordinary differential equation (ODE).  In a recent work \cite{Erbay2018}, this result has been improved in two ways. First,  the condition (\ref{order}) has been replaced by the ellipticity and boundedness condition
\begin{equation}
    c_{1}^{2}\leq \widehat{\beta }(\xi )\leq c_{2}^{2}  \label{bounds}
\end{equation}%
for some $c_{1}, c_{2}>0$. The condition (\ref{bounds}) implies that the convolution operator is invertible. On the other hand,  (\ref{bounds}) lacks the regularization effect of (\ref{order}) for $r \geq 2$. Hence the regularity requirement $r\geq 2$ in (\ref{order}) is replaced by the much weaker condition $r\geq 0$. Secondly, the long-time existence of solutions to  the family of initial-value problems of (\ref{origin}) has been established  for times up to ${\cal O}(1/\epsilon^{p})$.  We note that when $0\leq r<2,$ the nonlocal equation (\ref{nw}) is no longer an ODE but the parameter $\epsilon $ can be chosen small enough so that it is in the hyperbolic regime. Theorem 5.2 of  \cite{Erbay2018} is as follows:
\begin{theorem}\label{theo2.1}
    Suppose the kernel $\beta $ satisfies the ellipticity and boundedness condition (\ref{bounds}). Let $D>3$, $P>P_{min}$, $~s>\frac{7}{2}$ and $\left( u_{0}^{\epsilon },w_{0}^{\epsilon }\right) $ be bounded in $H^{s+P}\times H^{s+P}$. Then there exist some $\epsilon _{0}>0$, $~T>0$ and a unique family $(u^{\epsilon })_{0<\epsilon <\epsilon _{0}}$ bounded in $C\big( [0,{\frac{T}{\epsilon ^{p}}}];H^{s+D}\big) \cap C^{1}\big( [0,{\frac{T}{\epsilon ^{p}}}];H^{s+D-1}\big) $ and satisfying (\ref{origin}) with initial values $u_{0}=u_{0}^{\epsilon }$,
    $u_{1}=\big(w_{0}^{\epsilon }\big)_{x}$.
\end{theorem}
The numbers $P$ and $D$ of this theorem are related to the required extra smoothness and the restrictions on these numbers are due to the  Nash-Moser scheme used in the proof of Theorem  \ref{theo2.1} (see also \cite{Alvarez2008} for more details). For a given $D>3$, the number $P_{\min }$ is determined as $~P_{\min }=3+{\frac{D}{D-3}}\left(\sqrt{3}+\sqrt{2D}\right) ^{2}$.  The computation in \cite{Erbay2018} shows that the optimal values of $D$ and $P$ are approximately $7.35$ and  $55.34$, respectively.

Throughout  this work we will assume that the kernel  $\beta(x)$ is an even function with $\int \beta(x) dx=1$ and satisfies (\ref{bounds}).

In order to write the nonlocal equation (\ref{origin}) as a first-order system we now introduce the pseudo-differential operator
\begin{equation}
    Kw(x)=\mathcal{F}^{-1}\Big( k(\xi )\widehat{w}(\xi )\Big),  \label{K1}
\end{equation}%
where $k(\xi )=\sqrt{\widehat{\beta }(\xi )}$ and $\mathcal{F}^{-1}$ denotes the inverse Fourier transform. We note that $K^{2}w=\beta \ast w$ and, by  (\ref{bounds}), $K$ is an invertible bounded operator on $H^{s}$. We  then convert  (\ref{origin}) to
\begin{eqnarray}
    && u_{t} =Kv_{x},  \label{systema} \\
    && v_{t} =K\left(u+\epsilon^{p}u^{p+1}\right)_{x}. \label{systemb}
\end{eqnarray}
As mentioned earlier,   applying the coordinate transformation $(x,t)\to (x/\delta, t/\delta)$ to (\ref{origin}) yields (\ref{nw}) with $\beta _{\delta }(x)={\frac{1}{\delta }}\beta (\frac{x}{\delta })$ (note that   the Fourier domain counterpart of the latter is $\widehat{\beta _{\delta }}(\xi )=\widehat{\beta }(\delta\xi )$). Similarly, with the same coordinate transformation, the initial-value problem for (\ref{systema})-(\ref{systemb}) reduces to
\begin{eqnarray}
   && u_{t} =K_{\delta }v_{x},~~~~~~~~~~~~~~~~~~~~u(x,0)=u_{0}(x),  \label{system1} \\
   && v_{t} =K_{\delta }\left(u+\epsilon^{p}u^{p+1}\right)_{x},~~~~v(x,0)=v_{0}(x),  \label{system2}
\end{eqnarray}%
with the operator $K_{\delta }w=\mathcal{F}^{-1}\big(k_{\delta}(\xi)\widehat{w}(\xi)\big)$ and $k_{\delta}(\xi)=k(\delta\xi)$.  Clearly, when  $u_{1}=(w_{0})_{x}$ with $v_{0}=K_{\delta }^{-1}w_{0} $ the initial-value problem (\ref{nw})-(\ref{inidata}) becomes equivalent to  (\ref{system1})-(\ref{system2}); hence we have long time existence and uniform bounds for the solutions $\left(u^{\epsilon,\delta},v^{\epsilon,\delta}\right)$ of (\ref{system1})-(\ref{system2}). In fact the long-time existence result  in \cite{Erbay2018} was first proved for (\ref{system1})-(\ref{system2}) and then extended to (\ref{nw})-(\ref{inidata}).
\begin{remark}\label{rem2.2}
    The uniform bounds for the solution in Theorem \ref{theo2.1} depend only on $T$, the bounds on the initial data and the bound for the operator $K.$ Due to (\ref{bounds}), the family of operators $K_{\delta }$ all have the same bound as $K$; therefore Theorem \ref{theo2.1} also holds uniformly in $\delta$ when $K$ is replaced by $K_{\delta }$ or $\beta $ is replaced by $\beta_{\delta }$.
\end{remark}
\begin{corollary}\label{cor2.3}
    Suppose the kernel $\beta $ satisfies the ellipticity and boundedness condition (\ref{bounds}). Let $D>3$, $P>P_{min}$, $~s>\frac{7}{2}$ and $(\varphi ,\psi )$ $\in $ $H^{s+P}\times H^{s+P}$. Then there are some $\epsilon _{0}>0,$ $T>0$ so that for all $0<\epsilon <\epsilon_{0}$,  the family of Cauchy problems for (\ref{nw})  with initial values $\left(u_{0},u_{1}\right) $\ $=\left( \varphi ,\psi _{x}\right) $  have unique solutions $u^{\epsilon, \delta}$ on the interval $[0,{\frac{T}{\epsilon^{p}}}]$, uniformly bounded in $C\big( [0,{\frac{T}{\epsilon ^{p}}}];H^{s+D}\big) \cap C^{1}\big( [0,{\frac{T}{\epsilon ^{p}}}];H^{s+D-1}\big) $.
\end{corollary}

\section{Comparison of Solutions}\label{sec3}

In this section we prove the main result of this paper:  the difference between the  solutions of the Cauchy problems corresponding to two different kernel functions remains small over a long time interval in a suitable Sobolev norm if the kernels have the same dispersive nature in the long-wave limit.

Consider two different kernel functions $\beta^{(1)}$ and $\beta ^{(2)}$ satisfying the following three conditions for some $k\geq 1$:
\begin{itemize}
    \item[(C1)] $\beta^{(1)}$ and $\beta ^{(2)}$ satisfy the ellipticity and boundness condition (\ref{bounds}),
    \item[(C2)] $\beta^{(1)}$ and $\beta ^{(2)}$ have the same first $(2k-1)$-order moments, namely
        \begin{equation}
            \int x^{j}\beta ^{(1)}(x)dx=\int x^{j}\beta ^{(2)}(x)dx~~~\mbox{ for }~~0\leq j<2k-1,  \label{moments-real}
        \end{equation}
    \item[(C3)]  $x^{2k}\beta ^{(i)}(x)\in L^{1}\left( \mathbb{R}\right) $ ($i=1,2$).
\end{itemize}
Clearly in the case when $\beta =\mu $ is a finite measure, the moment integral should be replaced by $\int x^{j}d\mu$.
We consider (\ref{nw})-(\ref{inidata}) with $\beta^{(1)}$ and $\beta ^{(2)}$  and denote the corresponding solutions by $u_{1}^{\epsilon, \delta}$ and $u_{2}^{\epsilon, \delta}$, respectively. Our aim is to estimate the difference $u_{1}^{\epsilon, \delta}-u_{2}^{\epsilon, \delta}$.

Before proceeding to state the main result, we  introduce a lemma which provides   certain commutator estimates (see  Proposition B.8 of \cite{Lannes2013}).
\begin{lemma}\label{lem3.1}
    \noindent\ Let $q_{0}>1/2$, $~s\geq 0$ and let   $[\Lambda^s, w]g=\Lambda^s(wg)-w\Lambda^s g  $ with $\Lambda =\left( 1-D_{x}^{2}\right)^{1/2}$.
    \begin{enumerate}
    \item If \  $0\leq s\leq q_{0}+1$ and $w\in H^{q_{0}+1}$ then, for all $g\in H^{s-1}$, one has
        \begin{equation*}
            \Vert [ \Lambda^{s}, w]g\Vert_{L^{2}}\leq C\Vert w_{x}\Vert_{H^{q_{0}}}\Vert g\Vert_{H^{s-1}},
    \end{equation*}
    \item If $-q_{0}< r\leq q_{0}+1-s$ and $w\in H^{q_{0}+1}$ then, for all $g\in H^{r+s-1}$, one has
        \begin{equation*}
            \Vert [ \Lambda^{s}, w]g\Vert_{H^{r}}\leq C\Vert w_{x}\Vert_{H^{q_{0}}}\Vert g\Vert_{H^{r+s-1}}.
        \end{equation*}

    \end{enumerate}
\end{lemma}
We will use this lemma in finding the energy estimates for the system (\ref{r-sys})-(\ref{rho-sys}) to be introduced below. We are now ready to prove the main result.
\begin{theorem}\label{theo3.2}
    Let $\beta^{(1)}$ and $\beta^{(2)}$ be two kernels satisfying the conditions (C1), (C2) and (C3) for  some $k\geq 1$. Let $P$, $D$, $s$ be as in Theorem \ref{theo2.1} and $\varphi, \psi  \in H^{s+P+2k+1}$. Then there are some constants $\epsilon_{0}$, $C$ and $T>0$ \ independent of $\epsilon$ ($0<\epsilon < \epsilon_{0}$) and $\delta $ so that the solutions $u_{i}^{\epsilon,\delta }$ of the Cauchy problems
    \begin{eqnarray}
        && u_{tt} =\beta _{\delta }^{\left( i\right) }\ast (u+\epsilon ^{p}u^{p+1})_{xx},~~~~~x\in \mathbb{R}, ~~~~t>0,  \label{main-a} \\
        && u(x,0) =\varphi (x),~~~~u_{t}(x,0)=\psi _{x}(x),~~~~x\in \mathbb{R},  \label{main-b}
    \end{eqnarray}%
    for $\ i=1,2$ are defined for all $~t\in \left[ 0,\frac{T}{\epsilon ^{p}}\right] $ and satisfy
    \begin{equation}
        \Vert u_{1}^{\epsilon ,\delta }\left( t\right) -u_{2}^{\epsilon ,\delta}\left( t\right) \Vert _{H^{s+D}}
            \leq C\delta ^{2k}(1+t)\mbox{ \ \ for all \ }t\leq \frac{T}{\epsilon ^{p}}. \label{mainresult}
    \end{equation}
\end{theorem}
\begin{proof}
    We will complete the proof in several steps. For the rest of the proof we will drop the superscripts $\epsilon ,\delta $ to simplify the notation.
    \begin{description}
    \item[Step 1] Since $\varphi, \psi \in H^{s+P+2k+1}$, Theorem \ref{theo2.1} gives the uniform bound
        \begin{displaymath}
            \left\Vert u_{i}(t)\right\Vert _{H^{s+D+2k+1}}
                +\left\Vert (u_{i}) _{t}(t)\right\Vert _{H^{s+D+2k}}
                \leq C \mbox{ \ \ for all\ }t\leq \frac{T}{\epsilon ^{p}},~~~(i=1,2)
        \end{displaymath}
        for both families of solutions.
    \item[Step 2]
    Converting  (\ref{main-a})-(\ref{main-b}) into the corresponding systems of the form (\ref{system1})-(\ref{system2}) we obtain
    \begin{eqnarray*}
      &&  ( u_{i})_{t} =K_{\delta }^{(i)}( v_{i})_{x},~~~~u_{i}(x,0)=\varphi(x), \\
      &&  ( v_{i})_{t} =K_{\delta }^{(i)}\big(u_{i}+\epsilon^{p}u_{i}^{p+1}\big)_{x},
                ~~~~v_{i}(0)(x,0)=(K_{\delta}^{(i)})^{-1}\psi(x)
    \end{eqnarray*}%
    for $\ i=1,2$. Then the pair $(r, \rho)$, which are defined by the differences $r=u_{1}-u_{2}$ and $\rho =v_{1}-v_{2}$ between the solutions, satisfy:
    \begin{eqnarray}
       && r_{t} =K_{\delta }^{(2)}\rho _{x}+f_{1},\quad \quad  r(x,0)=0, \label{r-sys} \\
       && \rho_{t} =K_{\delta }^{(2)}r_{x}+\epsilon ^{p}K_{\delta }^{(2)}(wr)_{x}+f_{2},
                \quad \quad       \rho(x,0)=g(x),  \label{rho-sys}
    \end{eqnarray}%
    where
    \begin{eqnarray}
    && f_{1} =\big(K_{\delta }^{(1)}-K_{\delta }^{(2)}\big)(v_{1})_{x}, \\
    && f_{2} =\big(K_{\delta}^{(1)}-K_{\delta}^{(2)}\big)
            \big(u_{1}+\epsilon^{p}u_{1}^{p+1}\big)_{x}, \\
    && g     =\Big(\big(K_{\delta}^{(1)}\big)^{-1}-\big(K_{\delta }^{(2)}\big)^{-1}\Big)\psi, \\
    && w    =u_{1}^{p}+u_{1}^{p-1}u_{2}+\dots +u_{1}u_{2}^{p-1}+u_{2}^{p}.
    \end{eqnarray}%
    \item[Step 3]
        Conditions (C2) and (C3) mean that the Fourier transforms of the kernels satisfy $\widehat{\beta ^{(i)}}\in C^{2k}$ for $i=1,2$ and
        \begin{equation}
            \frac{d^{j}}{d\xi^{j}}\left( \widehat{\beta ^{(2)}}(\xi )-\widehat{\beta^{(1)}}(\xi )\right) \Big|_{\xi =0}=0~~\mbox{ for }~~0\leq j<2k-1. \label{moments-fourier}
        \end{equation}
        Then  $\widehat{\beta ^{(2)}}(\xi )-\widehat{\beta ^{(1)}}(\xi )=\mathcal{O}\left( \xi ^{2k}\right)$ and $k^{(2)}(\xi )-k^{(1)}(\xi )=\mathcal{O}\left( \xi ^{2k}\right) $ near the origin. Thus, we have
        \begin{displaymath}
            k^{(2)}(\xi )-k^{(1)}(\xi )=\xi ^{2k}m(\xi )
        \end{displaymath}%
        for some continuous function $m$. Moreover, as both $k^{(i)}(\xi )$ $(i=1,2)$ are bounded, then so is $m(\xi )$. Then the corresponding operators $K_{\delta}^{(i)}$ will satisfy
        \begin{displaymath}
            K_{\delta }^{(2)}=K_{\delta }^{(1)}+\left( -1\right) ^{k}\delta^{2k}D_{x}^{2k}M_{\delta },
        \end{displaymath}%
        where $M_{\delta }$ is the operator with symbol $m(\delta \xi )$. Since $m$ is bounded, we have the  estimate
        \begin{equation}
            \left\Vert \big(K_{\delta }^{(2)}-K_{\delta }^{(1)}\big)u\right\Vert_{H^{s}}
            = \delta^{2k}\left\Vert  D_{x}^{2k}M_{\delta }u\right\Vert _{H^{s}}
            \leq C\delta^{2k}\left\Vert u\right\Vert _{H^{s+2k}},  \label{K2K1est}
        \end{equation}%
        with the constant $C$ independent of $\delta $.
    \item[Step 4]
        We now estimate the terms $g$, $w$, $f_{1}$, $f_{2}$ appearing in (\ref{r-sys})-(\ref{rho-sys}).
        \begin{enumerate}
        \item Noting that
            \begin{eqnarray*}
                g&=&\Big(\big(K_{\delta }^{(1)}\big)^{-1}-\big(K_{\delta}^{(2)}\big)^{-1}\Big)\psi \\
                 &=&\big(K_{\delta }^{(1)}\big)^{-1}\big(K_{\delta}^{(2)}\big)^{-1}\big(K_{\delta }^{(2)}-K_{\delta }^{(1)}\big)\psi
            \end{eqnarray*}
             and that $\big(K_{\delta}^{(i)}\big)^{-1}$ $(i=1,2)$  are uniformly bounded, we have
            \begin{eqnarray*}
                \left\Vert g\right\Vert _{H^{s+P}}
                 &\leq & C\left\Vert \big( K_{\delta}^{(2)}-K_{\delta }^{(1)}\big)\psi \right\Vert _{H^{s+P}} \\
                &\leq & C\delta^{2k}\left\Vert \psi \right\Vert _{H^{s+P+2k}}\leq C\delta^{2k},
            \end{eqnarray*}
            where we have made use of (\ref{K2K1est}).
        \item From the definition $w=u_{1}^{p}+u_{1}^{p-1}u_{2}+\dots +u_{2}^{p}$, we get
           \begin{eqnarray}
              &&\left\Vert w(t)\right\Vert _{H^{s+D}}
                \leq C\sum_{i=0}^{p}\left\Vert u_{1}(t)\right\Vert _{H^{s+D}}^{p-i}\left\Vert u_{2}(t)\right\Vert_{H^{s+D}}^{i}
                \leq C, \label{estW}\\
              && \left\Vert w_{t}(t)\right\Vert _{H^{s+D-1}}
                \leq C\Big( \left\Vert \left(u_{1}\right) _{t}(t)\right\Vert _{H^{s+D-1}}
                +\left\Vert \left( u_{2}\right)_{t}(t)\right\Vert _{H^{s+D-1}}\Big) \leq C. \nonumber \\
                \label{estWt}
            \end{eqnarray}
        \item We have
            \begin{eqnarray}
                \left\Vert f_{1}(t)\right\Vert _{H^{s+D}}
                &=& \left\Vert \big(K_{\delta}^{(1)}-K_{\delta }^{(2)}\big)(v_{1})_{x}(t)\right\Vert_{H^{s+D}} \nonumber \\
                &\leq & C\delta ^{2k}\left\Vert \left( v_{1}\right)_{x}(t)\right\Vert _{H^{s+D+2k}} \nonumber \\
                &\leq & C\delta ^{2k}\left\Vert v_{1}(t)\right\Vert_{H^{s+D+2k+1}}
                    \leq C\delta ^{2k}. \label{estF1}
            \end{eqnarray}
        \item Similarly,
            \begin{eqnarray}
                \left\Vert f_{2}(t)\right\Vert _{H^{s+D}}
                &=&\left\Vert \big(K_{\delta}^{(1)}-K_{\delta}^{(2)}\big)\big(u_{1}
                +\epsilon^{p}u_{1}^{p+1}\big)_{x}(t)\right\Vert_{H^{s+D}} \nonumber \\
            &\leq &C\delta^{2k}\left\Vert \big(u_{1}
                +\epsilon^{p}u_{1}^{p+1}\big)_{x}(t)\right\Vert_{H^{s+D+2k}}  \nonumber   \\
            &\leq &C\delta ^{2k}\left\Vert \big(u_{1}
                +\epsilon^{p}u_{1}^{p+1}\big)(t)\right\Vert _{H^{s+D+2k+1}}
            \leq C\delta^{2k}. \nonumber \\ \label{estF2}
        \end{eqnarray}
\end{enumerate}
    \item[Step 5] Next we define the energy
    \begin{equation}
        E^{2}(t)=\frac{1}{2}\left( \left\Vert r(t)\right\Vert _{H^{s+D}}^{2}+\left\Vert\rho (t)\right\Vert _{H^{s+D}}^{2}+\epsilon ^{p}\big\langle r(t),w(t)r(t)\big\rangle_{s+D}\right), \label{Esquare}
    \end{equation}%
    with the $H^{s+D}$ inner product $\big\langle f, g\big\rangle_{s+D}=\big\langle \Lambda^{s+D}f, \Lambda^{s+D}g\big\rangle $. Since
    \begin{displaymath}
        \epsilon ^{p}\Big\vert \big\langle r(t),w(t)r(t)\big\rangle_{s+D}\Big\vert
        \leq \epsilon ^{p}\left\Vert w(t)\right\Vert_{H^{s+D}}\left\Vert r(t)\right\Vert _{H^{s+D}}^{2}
        \leq C\epsilon^{p}\left\Vert r(t)\right\Vert _{H^{s+D}}^{2},
    \end{displaymath}%
    there is some $\epsilon_{1}\leq \epsilon _{0}$ so that for all $\epsilon<\epsilon _{1}$,
    \begin{equation}
        \left\Vert r(t)\right\Vert _{H^{s+D}}^{2}+\left\Vert \rho (t)\right\Vert_{H^{s+D}}^{2}\leq C E^{2}(t)
         \label{E est}
    \end{equation}%
    (see Lemma 3.1 of \cite{Erbay2018}  for details).     Differentiation both sides of (\ref{Esquare}) with respect to $t$ gives
        \begin{eqnarray}
        && \!\!\!\!\!\!\!\!\!\!\!\!\!\!
        \frac{d~}{dt}E^{2}(t)
        =\big\langle r, r_{t}\big\rangle_{s+D}+\big\langle \rho, \rho_{t}\big\rangle_{s+D} \nonumber\\
        && ~~~~~~~ +\frac{\epsilon ^{p}}{2}\Big(\big\langle r_{t}, wr\big\rangle_{s+D}
            +\big\langle r, w_{t}r\big\rangle_{s+D}+\big\langle r, wr_{t}\big\rangle_{s+D}\Big).
    \end{eqnarray}
    By making use of the system (\ref{r-sys})-(\ref{rho-sys})  in this equation we get
    \begin{eqnarray}
      && \!\!\!\!\!\!\!\!\!\!\!\!\!\!
        \frac{d~}{dt}E^{2}(t)
        = \big\langle r,f_{1}\big\rangle _{s+D}+\big\langle\rho ,f_{2}\big\rangle _{s+D}
        +\big\langle \rho,\epsilon^{p}K_{\delta }^{(2)}( wr)_{x}\big\rangle _{s+D} \nonumber\\
        && ~~~~~~ +\frac{\epsilon ^{p}}{2}\Big( \big\langle r_{t}, wr\big\rangle_{s+D}
            +\big\langle r, w_{t}r\big\rangle_{s+D}+\big\langle r, wr_{t}\big\rangle_{s+D}\Big), \label{E2time}
    \end{eqnarray}
    where we have made use of the identity $\big\langle r, K_{\delta }^{(2)}\rho_{x}\big\rangle_{s+D}=-\big\langle \rho, K_{\delta}^{(2)}r_{x}\big\rangle_{s+D}$.     Using (\ref{r-sys}) the third term on the right-hand side of this equation can be written as
     \begin{eqnarray}
    \big\langle \rho,\epsilon^{p}K_{\delta }^{(2)}( wr)_{x}\big\rangle _{s+D}
        &=& -\epsilon^{p}\big\langle K_{\delta }^{(2)}\rho_{x},wr\big\rangle_{s+D}          \nonumber\\
        &=& -\epsilon^{p}\big\langle r_{t},wr\big\rangle_{s+D}+\epsilon^{p}\big\langle f_{1},wr\big\rangle_{s+D}
    \end{eqnarray}
    Substitution of  this result into  (\ref{E2time}) yields
     \begin{eqnarray}
        && \!\!\!\!\!\!\!\!\!\!\!\!\!\!
        \frac{d~}{dt}E^{2}(t)
        =\big\langle r, f_{1}\big\rangle_{s+D}+\big\langle \rho, f_{2}\big\rangle_{s+D}
            +\epsilon^{p}\big\langle wr, f_{1}\big\rangle_{s+D}
            +\frac{\epsilon^{p}}{2}\big\langle r, w_{t}r\big\rangle_{s+D} \nonumber\\
        && ~~~~~~~ +\frac{\epsilon^{p}}{2}\Big( \big\langle r, w r_{t}\big\rangle_{s+D}
            -\big\langle r_{t}, wr\big\rangle_{s+D}\Big). \label{energy}
        \end{eqnarray}
    Noting that $\Lambda^{s+D}(wf)=[\Lambda^{s+D}, w]f+w\Lambda^{s+D} f$ we can rewrite the last term of (\ref{energy}) as
    \begin{eqnarray}
         \big\langle r, w r_{t}\big\rangle_{s+D}-\big\langle r_{t}, wr\big\rangle_{s+D}
        &=& \big\langle \Lambda^{s+D}r, \lbrack \Lambda^{s+D},w\rbrack r_{t}\big\rangle \nonumber \\
          &&  -\big\langle \Lambda^{s+D-1}r_{t}, \Lambda \lbrack \Lambda^{s+D},w\rbrack r\big\rangle.
    \end{eqnarray}%
     To estimate this term we start with
    \begin{eqnarray}
        \Big\vert \big\langle r, w r_{t}\big\rangle_{s+D}-\big\langle r_{t}, wr\big\rangle_{s+D}\Big\vert
        &\leq &C\Big(\big\Vert \Lambda^{s+D}r\big\Vert_{L^{2}}\big\Vert\lbrack \Lambda^{s+D},w]r_{t}\big\Vert_{L^{2}} \nonumber \\
        && +\big\Vert \Lambda^{s+D-1}r_{t}\big\Vert_{L^{2}}\big\Vert \Lambda \lbrack \Lambda^{s+D},w]r\big\Vert_{L^{2}}
        \Big ). \nonumber \\ \label{thirdt}
    \end{eqnarray}%
     On the other hand, using the commutator estimates given in Lemma \ref{lem3.1} we get
    \begin{eqnarray*}
       \big \Vert [\Lambda^{s+D},w]r_{t}\big\Vert_{L^{2}}
            &\leq  & C \Vert w\Vert_{H^{s+D+1}} \Vert r_{t} \Vert_{H^{s+D-1}},\\
             &\leq  & C \Vert w\Vert_{H^{s+D+1}}\Big ( \Vert \rho\Vert_{H^{s+D}} + \Vert f_1 \Vert_{H^{s+D-1}} \Big ),\\
        \big\Vert \Lambda [\Lambda^{s+D},w]r\big\Vert_{L^{2}}
            & = & \big\Vert [\Lambda^{s+D},w]r\big\Vert_{H^{1}}
            \leq  C \Vert w \Vert_{H^{s+D+1}} \Vert r\Vert_{H^{s+D}},
    \end{eqnarray*}
    and
    \begin{eqnarray*}
       \big \Vert \Lambda^{s+D-1} r_t\big \Vert_{L^{2}}
            & \leq &  C\Big (\big\Vert  \Lambda^{s+D-1} K_{\delta}^{(2)}\rho _{x}\big\Vert _{L^2}
            + \big\Vert \Lambda^{s+D-1}  f_1\big\Vert_{L^2}  \Big )\\
     &\leq & C\Big (\Vert \rho\Vert _{H^{s+D}}+  \Vert f_1\Vert_{H^{s+D-1}}  \Big ).
    \end{eqnarray*}
    By making use of  these results in (\ref{thirdt}) we obtain
    \begin{eqnarray}
      && \hspace*{-40pt} \Big\vert \big\langle r, w r_{t}\big\rangle_{s+D}
        -\big\langle r_{t}, wr\big\rangle_{s+D}\Big\vert \nonumber \\
      && ~~~~~  \leq  C  \Vert r\Vert _{H^{s+D}}  \Vert w\Vert _{H^{s+D+1}} \Big (\Vert \rho \Vert_{H^{s+D}}
        + \Vert f_1\Vert _{H^{s+D-1}} \Big ).
        \label{third}
    \end{eqnarray}
    Combining (\ref{third}) with  (\ref{energy}) yields
     \begin{eqnarray*}
     && \!\!\!\!\!\!\!\!\!\!\!\!\!\!\!\!
        \frac{d~}{dt}E^{2}(t)
       \leq C \Big (\Vert r\Vert _{H^{s+D}}\left\Vert f_{1}\right\Vert_{H^{s+D}}
            +\left\Vert \rho \right\Vert _{H^{s+D}}\left\Vert f_{2}\right\Vert_{H^{s+D}}
            \Big )\\
        && ~~~~~ +C\epsilon ^{p} \Vert r\Vert _{H^{s+D}} \Big (\Vert w\Vert _{H^{s+D}}\Vert f_1\Vert _{H^{s+D}}
        +\Vert w\Vert _{H^{s+D+1}}\left\Vert f_{1}\right\Vert _{H^{s+D-1}} \Big )\\
        && ~~~~~ +C\epsilon ^{p} \Vert r\Vert _{H^{s+D}}\Big (\Vert r\Vert _{H^{s+D}}\Vert w_{t}\Vert _{H^{s+D}}
                + \Vert \rho \Vert _{H^{s+D}}\Vert w\Vert _{H^{s+D+1}}   \Big ).
    \end{eqnarray*}%
     Using the estimates on  $\left\Vert w\right\Vert _{H^{s+D}}$,  $\left\Vert w_{t}\right\Vert_{H^{s+D}}$,   $\left\Vert f_{i}\right\Vert _{H^{s+D}}$  $(i=1,2)$ given by (\ref{estW}), (\ref{estWt}), (\ref{estF1}), (\ref{estF2}), respectively, we get
      \begin{eqnarray*}
     && \!\!\!\!\!\!\!\!\!\!\!\!\!\!\!\!
        \frac{d~}{dt}E^{2}(t)
       \leq C \delta^{2k}\Big (\Vert r\Vert _{H^{s+D}}+\Vert \rho \Vert _{H^{s+D}} \Big )
        +C\epsilon ^{p}\delta^{2k} \Vert r\Vert _{H^{s+D}}         \\
        && ~~~~~ +C\epsilon ^{p} \Vert r\Vert _{H^{s+D}}\Big (\Vert r\Vert _{H^{s+D}}
                + \Vert \rho \Vert _{H^{s+D}}   \Big ).
    \end{eqnarray*}%
    and
    \begin{displaymath}
       \frac{d~}{dt}E^{2}(t)\leq C\Big( \delta ^{2k}E(t)+\epsilon ^{p}E^{2}(t)\Big) .
    \end{displaymath}%
    Gronwall's lemma then implies
    \begin{equation}
        E(t)\leq e^{C\epsilon^{p}t}E(0)+\frac{e^{C\epsilon ^{p}t}-1}{\epsilon^{p}}\delta^{2k}
        \leq e^{CT}\big( E(0)+C\delta ^{2k}t\big) \label{gronwall}
    \end{equation}%
    for all $t\leq \frac{T}{\epsilon ^{p}}$. Equation (\ref{Esquare}) shows that the initial energy is
    \begin{displaymath}
        E(0)
        =\frac{1}{\sqrt{2}}\Big( \left\Vert r(0)\right\Vert_{H^{s+D}}^{2}
            +\left\Vert\rho (0)\right\Vert _{H^{s+D}}^{2}
            +\epsilon ^{p}\big\langle r(0),w(0)r(0)\big\rangle _{s+D}\Big)^{1/2}
    \end{displaymath}%
    from which we get
    \begin{displaymath}
        E(0)
        =\frac{1}{\sqrt{2}}\left\Vert g\right\Vert_{H^{s+D}}\leq C\left\Vert g\right\Vert_{H^{s+P}}\leq C\delta ^{2k}.
    \end{displaymath}%
    By making use of this result in (\ref{gronwall}) we obtain (\ref{mainresult}) and this completes the proof of the theorem.
    \end{description}
\end{proof}
\begin{remark}\label{rem3.3}
    In this remark, we show that  solutions of (\ref{nw}) for suitable kernels are well approximated by solutions of a high order improved Boussinesq-type equation. Suppose that  the kernel $\beta $  satisfies $x^{2k}\beta \in L^{1}\left( \mathbb{R}\right)$ and the ellipticity and boundedness condition (\ref{bounds}). We now assume that the  $2k$-order Taylor polynomial of $\left( \widehat{\beta}(\xi )\right)^{-1}$ about zero is as follows
    \begin{displaymath}
        \left( \widehat{\beta_{0}}(\xi )\right)^{-1}=1+\gamma _{1}\xi ^{2}+\gamma _{2}\xi ^{4}+\cdots+\gamma _{k-1}\xi ^{2(k-1)},
    \end{displaymath}
    with suitable constants $\gamma _{j}$ ($j=1, \cdots ,k-1)$. Clearly $\widehat{\beta_{0}}$ satisfies the ellipticity and boundedness condition (\ref{bounds}). The wave equation corresponding to the kernel $\beta_{0}$ is     the high order improved Boussinesq-type equation
    \begin{equation}
        Lu_{tt}-u_{xx}
        =\epsilon ^{p}(u^{p+1})_{xx} \label{highIBq}
    \end{equation}%
    where
    \begin{displaymath}
        L=I-\gamma _{1}\delta^{2}D_{x}^{2}+\gamma _{2}\delta^{4}D_{x}^{4}
            +\cdots +(-1)^{k-1}\gamma_{k-1}\delta^{2k-2}D_{x}^{2k-2}.
    \end{displaymath}%
    We now  compare  solutions of (\ref{nw}) to  solutions of (\ref{highIBq}).  In that respect we note that  the Taylor expansions for  $\left(\widehat{\beta}(\xi)\right)^{-1}$ and $\left(\widehat{\beta_{0}}(\xi)\right)^{-1}$ around the origin agree on all terms of degree less than $2k$.   Applying Theorem \ref{theo3.2} for the pair $(\beta,\beta_{0})$ we conclude that  solutions of (\ref{nw}) are well approximated by solutions of (\ref{highIBq}) with an approximation error of order $\delta^{2k}$. More explicitly, for sufficiently smooth initial conditions the solutions  corresponding to the pair $(\beta,\beta_{0})$ approximate each other at the order $\delta^{2k}$ over times $t\leq \frac{T}{\epsilon^{p}}$.

    In particular, for the improved Boussinesq equation (\ref{ib}) we have $k=2$, $\gamma _{1}=1$ and $\left( \widehat{\beta_{0}}(\xi )\right) ^{-1}=1+\xi ^{2}$. When we compare solutions of  (\ref{nw}) with those of (\ref{ib}), Theorem \ref{theo3.2}  will yield an approximation error of order $\delta^{4}$.
\end{remark}
\begin{remark}\label{rem3.4}
    Here we explain why the equality of the lowest-order moments of the kernels is more important than the similarity of the shapes of the kernels when comparing two nonlocal wave equations in the long-wave limit.  To this end, we consider three types of general perturbations of an arbitrary kernel function $\beta_{0}$, and apply Theorem \ref{theo3.2} to investigate the closeness of the corresponding solutions of (\ref{nw}). With some appropriate kernel function $\varphi$, we define the three perturbations in the form
    \begin{eqnarray*}
         \beta_{1} &=&(1-\nu )\beta_{0}+\nu \varphi , \\
         \beta_{2} &=&\varphi_{\nu }\ast \beta_{0},  \\
         \beta_{3} &=&\beta_{0}+aD_{x}^{2k}\varphi,
    \end{eqnarray*}%
     where $\varphi_{\nu}(x)=\frac{1}{\nu}\varphi (\frac{x}{\nu})$, $\nu \ll 1$ and $a$ are  nonzero  constants. We denote the corresponding solutions by $u_{i}^{\epsilon, \delta}$ (i=0,1,2,3). Clearly both $\beta_{1}$ and $\beta_{2}$ are classical small perturbations of $\beta_{0}$ whereas $\beta_{0}-\beta_{3}={\cal O}(a)$ can be quite large. On the other hand, if $\int x^{2}\varphi(x)dx\not=\int x^{2}\beta_{0}(x)dx$, then $\beta_{0}$ and $\beta_{1}$ will have unequal second moments. Then for the corresponding solutions of (\ref{nw}),  Theorem \ref{theo3.2} will only guarantee the estimate $u_{0}^{\epsilon,\delta}-u_{1}^{\epsilon,\delta }={\cal O}(\delta^{2})$ over long times. Similarly, when $\int x^{2}\varphi (x)dx\not=0$, the kernels $\beta_{0}$ and $\beta_{2}$ will have unequal second moments; hence  $u_{0}^{\epsilon,\delta}-u_{2}^{\epsilon,\delta }={\cal O}(\delta^{2})$. When $\varphi \in W^{2k,1},$ $\beta_{0}$ and $\beta _{3}$ will have equal moments up to order $2k-1$.  For $k\geq 2$, Theorem \ref{theo3.2} this time implies the much stronger estimate $u_{0}^{\epsilon,\delta }-u_{3}^{\epsilon,\delta }={\cal O}(\delta^{2k})$. These three examples show that  the behavior of solutions of (\ref{nw}) over large times is determined by the dispersive character of the kernel in the long-wave limit rather than its shape.
\end{remark}
\begin{remark}\label{rem3.5}
It is worth to note that, under the change of variable $\widetilde{u}=\epsilon u$, the Cauchy problem (\ref{nw})-(\ref{inidata}) becomes
\begin{eqnarray*}
    && \widetilde{u}_{tt}=\beta _{\delta}\ast (\widetilde{u}+\widetilde{u}^{p+1})_{xx},~~~~x\in \mathbb{R},~~~t>0,\\
    && \widetilde{u}(x,0)=\epsilon u_{0}(x), ~~~~ \widetilde{u}_{t}(x,0)=\epsilon u_{1}(x), ~~~~x\in \mathbb{R},
\end{eqnarray*}%
which shows that our results can also be interpreted as comparison results for the long-wave approximations of solutions with small initial data.
\end{remark}

\bibliographystyle{tfnlm}
\bibliography{references}

\end{document}